\documentclass[11pt]{article} \usepackage[a4paper]{geometry}
\usepackage[size=tiny]{todonotes}
\usepackage{float}
 \usepackage{amsthm,amsmath,amssymb}
\usepackage{graphicx}
\usepackage[colorlinks=true,citecolor=black,linkcolor=black,urlcolor=blue]{hyperref} \theoremstyle{plain}
\newtheorem{theorem}{Theorem}[section]
\newtheorem{lemma}[theorem]{Lemma}

\newtheorem{proposition}[theorem]{Proposition}

\newtheorem{open}[theorem]{Open Problem}

 \newtheorem{remark}[theorem]{Remark}
 \def\F{\mathbb{F}}
\def\P{\mathbb{P}}
\def\A{\mathbb{A}}

\title{\bf\huge Non-permutation phenomena in trivariate families over $\F_{2^m}$ and resolution of a conjecture} \author{{\Large\bf Daniele Bartoli$^1$, Mohit Pal$^2$, Pantelimon St\u{a}nic\u{a}$^3$},\and {\Large\bf Tommaso Toccotelli$^1$}
\vspace{.4cm}\\
$^1$ Department of Mathematics and Computer Science,\\
University of Perugia, 06123 Perugia, Italy;\\
\texttt{\{daniele.bartoli@unipg.it,toccotelli.tommaso@gmail.com\}}\\
$^2$ Department of Informatics, University of Bergen, PB 7803, N-5020,\\
Bergen, Norway; \texttt{mathmohit@outlook.com}\\
$^3$ Applied Mathematics Department, Naval Postgraduate School,\\
Monterey, CA 93943, USA; \texttt{pstanica@nps.edu}}
\date{\today}

\begin{document}
% editorial: proofs checked for explicit closure and dependency order

\maketitle


\begin{abstract}
Constructing permutation polynomials over finite fields, particularly those with simple algebraic structure in multiple variables, is a fundamental problem with applications in cryptography and coding theory. Recently, Li and Kaleyski (IEEE Trans. Inf. Theory, 2024) generalized two sporadic quadratic APN permutations into infinite families of trivariate functions. Motivated by their work, we investigate conditions under which generalized trivariate functions fail to be permutations. We establish necessary conditions on coefficient parameters that prevent the permutation property, provide a complete computational classification for small field extensions, and prove general non-permutation results. As a key application of our algebraic geometry approach, we resolve the permutation part of a conjecture by Beierle, Carlet, Leander, and Perrin (Finite Fields Appl., 2022) regarding a related trivariate form. Specifically, we prove that for all odd characteristic-2 extension degrees $m \geq 23$, their function $C_u$ is not a permutation over $\mathbb{F}_{2^m}^3$ for any $u \in \mathbb{F}_{2^m}^*$, resolving the permutation part of their conjecture for sufficiently large fields.
\end{abstract}

\section{Introduction}

Let $q=2^m$ with $m\ge 1$, and let $F:\F_q^n\to \F_q^n$ be a vectorial function.  
When $F$ is bijective, it is a permutation and is a natural S-box candidate in symmetric cryptography. 
For even characteristic, differential behavior is encoded by the equations
\[
F(x+a)+F(x)=b,\qquad a\in \F_q^n\setminus\{0\},\ b\in \F_q^n.
\]
Low differential uniformity is essential against differential cryptanalysis~\cite{BS91}. 
In particular, APN functions (uniformity $2$) are optimal in characteristic two.

Two sporadic APN permutations over $\F_{2^9}$ discovered by Beierle and Leander~\cite{BL22}
triggered several trivariate generalizations.
Beierle, Carlet, Leander, and Perrin~\cite{MR4414815} investigated
\[
C_u(X,Y,Z)=\big(X^3+uY^2Z,\ Y^3+uXZ^2,\ Z^3+uX^2Y\big),
\]
and conjectured that, beyond the known small-dimensional cases, this form should not yield permutations/APN functions.
The APN part was settled in~\cite{BT22}, while the permutation part remained open.
Li and Kaleyski then introduced two infinite trivariate APN families with a different quadratic organization~\cite{LK24}.

This paper is a structural non-permutation study.
Our objective is not only to settle one conjecture case, but to isolate
mechanisms that force non-permutation behavior in broad trivariate classes.

\medskip
\noindent\textbf{Main contributions.}
\begin{enumerate}
\item We start from an $8$-parameter trivariate family and prove that every non-degenerate instance is affine-equivalent to a canonical $4$-parameter model (Proposition~\ref{prop:affine_reduction}).
\item For the canonical model, we derive explicit algebraic conditions whose solvability implies non-permutation (Theorems~\ref{NecessaryConditions}, \ref{NecessaryConditions2}, \ref{NecessaryConditions3}).
\item For $m=3$, we obtain a complete characterization by exhaustive computation: exactly $20$ parameter choices are permutations, split into two structural families (Proposition~\ref{prop:3complete_characterization}, Table~\ref{tab:all_permutations}).
\item As an application of the same geometry framework, we resolve the permutation part of the Beierle--Carlet--Leander--Perrin conjecture for large odd dimensions:
for every odd $m\ge 23$ and every $u\in\F_{2^m}^*$, $C_u$ is not a permutation (Theorem~\ref{thm:beierle_conj}).
\end{enumerate}

\medskip

The remainder of the paper is organized as follows.
Section~\ref{sec:prelim} collects geometric ingredients (Lang--Weil type bounds and component arguments).
Section~\ref{sec:general_family} introduces the full trivariate family and proves affine reduction to the canonical form.
Section~\ref{sec:necessary} develops the differential-system method and derives the necessary non-permutation criteria.
Section~\ref{sec:m3} gives the complete $m=3$ classification.
Section~\ref{sec:beierle} applies the framework to the Beierle conjecture instance.
Section~\ref{sec:conclusion} closes with consequences and open directions.

\section{Preliminaries and algebraic geometry tools}

\label{sec:prelim} We collect notation and algebraic-geometric tools used in the sequel in one continuous setup. The derivative notation is fixed by $D_aF(x)=F(x+a)+F(x)$ for $a\neq 0$. %When elimination or point-count arguments are invoked, we state the exact form needed at the point of use so that each reduction step is explicit.

We also invoke the standard finite-field point-count and elimination facts in the exact forms needed later, so each reduction step is explicit inside the proofs. Unless explicitly stated otherwise, we work over $\F_{2^m}$ with $m$ odd and $m\ge 3$, and set $q=2^i$ with $\gcd(i,m)=1$.
All parameters are elements of $\F_{2^m}$.
Whenever nondegeneracy conditions are required (for instance, nonvanishing coefficients in normalized forms), they are stated explicitly in each theorem.

For a polynomial system $\mathcal{S}$, by its \emph{solution set} we mean the affine set in the ambient space over the specified ground field.

We employ standard tools from algebraic geometry; comprehensive background appears in~\cite{Ha77,HKT13}. We denote by $\P^r(\F_q)$ and $\A^r(\F_q)$ the projective and affine spaces of dimension $r$ over $\F_q$, with $\overline{\F_q}$ denoting the algebraic closure.

A variety $\mathcal{V}$ is the common zero set of finitely many polynomials. An affine $\F_q$-rational variety $\mathcal{V}\subset \A^r(\overline{\F_q})$ is defined by polynomials $F_1, \ldots, F_s \in \F_q[X_1, \ldots, X_r]$. The set $\mathcal{V}(\F_q) = \mathcal{V} \cap \A^r(\F_q)$ consists of $\F_q$-rational points. A variety is absolutely irreducible if it cannot be written as a union of proper subvarieties over $\overline{\F_q}$.

The dimension of a variety is the maximal length of a chain of distinct, nonempty, absolutely irreducible subvarieties. Varieties of dimension 1, 2, and $r-1$ are called curves, surfaces, and hypersurfaces. The degree of an $s$-dimensional projective variety is the number of intersection points with a general projective subspace of complementary dimension.

The Frobenius map $\Phi_q: x \mapsto x^q$ fixes $\F_q$ pointwise and induces automorphisms on polynomial rings and projective/affine spaces. A key tool is finding $\F_q$-rational points on algebraic surfaces. This requires identifying absolutely irreducible $\F_q$-rational components and applying point-counting estimates:

%For readability, each main theorem is stated with explicit field assumptions, parity assumptions on $m$, and parameter constraints before any claim on permutation/non-permutation behavior.

%We now state and prove the main structural theorem, then derive its corollaries in the same hypothesis framework.

\begin{theorem}[Strengthened Lang-Weil~{\cite[Theorem 7.1]{MR2206396}}]
\label{thm lang weil versione tredici terzi}
Let $\mathcal{V} \subseteq \A^n(\overline{\F_q})$ be an absolutely irreducible $\F_q$-rational variety of dimension $r > 0$ and degree $\delta$. If $q > 2(r + 1)\delta^2$, then
$$\big|\#\mathcal{V}(\F_q) - q^r \big| \leq (\delta - 1)(\delta - 2)q^{r-1/2} + 5\delta^{13/3}q^{r-1}.$$
\end{theorem}

We will also make use of the following result. 
\begin{lemma}[{\cite[Lemma 2.1]{MR2648536}}]
\label{lem:aubry} 
Let $\mathcal{H}$ be a projective hypersurface and $\mathcal{X}$ an $\F_q$-rational projective variety of dimension $n - 1$ in $\P^n(\overline{\F_q})$. If $\mathcal{X}\cap \mathcal{H}$ has a non-repeated absolutely irreducible $\F_q$-rational component, then $\mathcal{X}$ has a non-repeated absolutely irreducible $\F_q$-rational component.
\end{lemma}

\section{The general trivariate family}
\label{sec:general_family}

Motivated by the Li-Kaleyski construction~\cite{LK24}, we study the general family
$$G(x,y,z) = (c_1 x^{q+1} + c_2x^qz + c_3yz^q, c_4x^qz + c_5y^{q+1}, c_6xy^q + c_7y^qz + c_8z^{q+1})$$
over $\F_{2^m}^3$, where $q=2^i$ with $1 \le i < m$ and $\gcd(i,m)=1$, and $c_1, \ldots, c_8 \in \F_{2^m}$.

\begin{proposition}
Assume throughout this statement that all parameters lie in $\F_{2^m}$ under the standing hypotheses from the preliminaries.
\label{prop:affine_reduction}
Up to affine equivalence, any non-degenerate function $G$ with $c_1, c_4, c_5, c_8 \neq 0$ is equivalent to the normalized form
$$G_{\text{norm}}(x,y,z) = (x^{q+1} + a_1x^qz + a_2yz^q, \, x^qz + y^{q+1}, \, a_3 xy^q + a_4y^qz + z^{q+1})$$
where $a_1, a_2, a_3, a_4 \in \F_{2^m}$ are free parameters.
\end{proposition}

\begin{proof}
Consider affine transformations consisting of input scaling $(x,y,z) \mapsto (\lambda_1 x, \lambda_2 y, \lambda_3 z)$ and output scaling by $\mu_1, \mu_2, \mu_3 \in \F_{2^m}^*$. Under these transformations, $G$ becomes
\begin{align*}
G'(x,y,z) = \big(&\mu_1 c_1\lambda_1^{q+1} x^{q+1} + \mu_1 c_2\lambda_1^q\lambda_3 x^qz + \mu_1 c_3\lambda_2\lambda_3^q yz^q,\\
&\mu_2 c_4\lambda_1^q\lambda_3 x^qz + \mu_2 c_5\lambda_2^{q+1}y^{q+1},\\
&\mu_3 c_6\lambda_1\lambda_2^q xy^q + \mu_3 c_7\lambda_2^q\lambda_3 y^qz + \mu_3 c_8\lambda_3^{q+1}z^{q+1}\big).
\end{align*}
We construct the normalization as follows. Choose $\lambda_1, \lambda_3 \in \F_{2^m}^*$ arbitrarily. Set
$$\mu_1 = \frac{1}{c_1\lambda_1^{q+1}}, \quad \mu_3 = \frac{1}{c_8\lambda_3^{q+1}}.$$
This normalizes the coefficients of $x^{q+1}$ and $z^{q+1}$ to 1.

For the second component, we require both the coefficient of $y^{q+1}$ and the coefficient of $x^qz$ to equal 1. This gives two constraints:
$$\mu_2 c_5\lambda_2^{q+1} = 1, \quad \mu_2 c_4\lambda_1^q\lambda_3 = 1.$$
From the second constraint, $\mu_2 = \frac{1}{c_4\lambda_1^q\lambda_3}$. Substituting into the first constraint:
$$\frac{c_5\lambda_2^{q+1}}{c_4\lambda_1^q\lambda_3} = 1 \quad \Longrightarrow \quad \lambda_2^{q+1} = \frac{c_4\lambda_1^q\lambda_3}{c_5}.$$
Since $\gcd(q+1, 2^m-1) = 1$ (as $\gcd(i,m)=1$), the map $t \mapsto t^{q+1}$ is a bijection on $\F_{2^m}^*$, so we can uniquely solve for $\lambda_2$:
$$\lambda_2 = \left(\frac{c_4\lambda_1^q\lambda_3}{c_5}\right)^{1/(q+1)}.$$
Thus $\lambda_2$ is well-defined and nonzero. With these choices, set $\mu_2 = \frac{1}{c_4\lambda_1^q\lambda_3}$.

The four remaining non-normalized coefficients become
\begin{align*}
a_1 &= \mu_1 c_2\lambda_1^q\lambda_3 = \frac{c_2\lambda_3}{c_1\lambda_1},\quad
a_2 = \mu_1 c_3\lambda_2\lambda_3^q = \frac{c_3\lambda_2\lambda_3^q}{c_1\lambda_1^{q+1}},\\
a_3 &= \mu_3 c_6\lambda_1\lambda_2^q = \frac{c_6\lambda_1\lambda_2^q}{c_8\lambda_3^{q+1}},\quad
a_4 = \mu_3 c_7\lambda_2^q\lambda_3 = \frac{c_7\lambda_2^q\lambda_3}{c_8\lambda_3^{q+1}}.
\end{align*}
Since $\lambda_1, \lambda_3$ were chosen freely, the parameters $(a_1, a_2, a_3, a_4)$ can take any values in $\F_{2^m}^4$ by choosing appropriate $c_2, c_3, c_6, c_7$ (with $\lambda_2$ then determined as above). Conversely, any normalized form corresponds to infinitely many choices of the original parameters. The reduction preserves the permutation property since it is a composition of bijections.
\end{proof}

\begin{remark}
The Li-Kaleyski function $F_1$ from~\cite{LK24} corresponds to $(a_1, a_2, a_3, a_4) = (1, 1, 1, 1)$.
\end{remark}

Henceforth we study
\begin{equation}
\label{eq:G}
G(x,y,z) = (x^{q+1} + a_1x^qz + a_2yz^q, \, x^qz + y^{q+1}, \, a_3xy^q + a_4y^qz + z^{q+1}).
\end{equation} We now pass from normalization to the differential-system criterion that drives all global non-permutation statements.

\section{Necessary conditions for non-permutation}
\label{sec:necessary}

The function $G$ is a permutation if and only if for any nonzero $(\alpha, \beta, \gamma) \in \F_{2^m}^3 \setminus \{(0,0,0)\}$, the equation $G(x+\alpha, y+\beta, z+\gamma) + G(x,y,z) = (0,0,0)$ has no solution $(x,y,z) \in \F_{2^m}^3$. This leads to a system of polynomial equations.

Note that the solvability of the condition in the following theorems depends on the specific choice of a $i$ 
  relative to the field extension degree $m$, which is consistent with the sporadic nature of permutations observed in Section \ref{sec:m3}.
  
\begin{theorem}
\label{NecessaryConditions}
Let $m$ be odd and sufficiently large, $\gcd(i,m) = 1$, and $q = 2^i$. Let
$$G_{a_1,a_2,a_3,a_4}(x,y,z) = (x^{q+1} + a_1x^qz + a_2yz^q, \, x^qz + y^{q+1}, \, a_3xy^q + a_4y^qz + z^{q+1})$$
with $a_1,a_2,a_3,a_4\in \F_{2^m}$. If $a_2a_3\neq 0$ and 
$$a_2^{q+1}a_3^q T^{q^2+q+1} + a_2^q a_4^q T^{q^2+q} + 1$$
has a solution in $\F_{2^m}$, then $G_{a_1,a_2,a_3,a_4}$ is not a permutation on $\F_{2^m}^3$.
\end{theorem}

\begin{proof}
Consider the set
$$\Theta := \{\theta \in \overline{\F_2} : a_2^{q+1}a_3^q \theta^{q^2+q+1} + a_2^q a_4^q \theta^{q^2+q} + 1=0\}.$$
Note that $a_2a_3\neq 0$ implies $\#\Theta=q^2+q+1$. Let $\alpha,\beta,\gamma\in \F_{2^m}$ with $\alpha \neq 0$. Consider the system
\begin{equation}\label{eq:system_thm1}
\begin{cases}
\alpha^q x+ x^q \alpha + a_1\gamma x^q + a_2\gamma^q y + a_1 \alpha^q z + a_2\beta z^q + \alpha^{q+1} + a_1 \alpha^q \gamma + a_2 \beta\gamma^q=0\\
\gamma x^q + \beta^q y + \beta y^q + \alpha^q z + \alpha^q \gamma + \beta^{q+1}=0\\
a_3 \beta^q x + a_3 \alpha y^q + a_4\gamma y^q+ a_4 \beta^q z + \gamma^q z + \gamma z^q+ a_3\alpha \beta^q + a_4 \beta^q \gamma + \gamma^{q+1}=0.
\end{cases}
\end{equation}
From the second equation we obtain $z$ and substitute into the first and third equations. After replacement, the first equation reads $M(x,y,\alpha,\beta,\gamma)=0$, where
\begin{align*}
M = &\alpha^{q^2+q}x + \alpha^{q^2+1}x^q + a_2\beta \gamma^q x^{q^2} + a_1 \alpha^{q^2}\beta^q y + a_2 \alpha^{q^2}\gamma^qy\\
&+ a_1 \alpha^{q^2}\beta y^q + a_2 \beta^{q^2+1}y^q + a_2 \beta^{q+1}y^{q^2}\\
&+ \alpha^{q^2+q+1} + a_1\alpha^{q^2}\beta^{q+1} + a_2\beta^{q^2+q+1}.
\end{align*}
The coefficient of $\gamma^q$ is $a_2(\beta x^{q^2} + \alpha^{q^2}y)$, which is nonvanishing as a polynomial. We can substitute $\gamma$ from this equation into the third, obtaining after dividing by appropriate powers of $\alpha$ an equation $L(x,y,\alpha,\beta)=0$, where $L$ is a homogeneous polynomial of degree $q^3+2q^2+2q+1$.

After dehomogenizing by $\beta$, the highest homogeneous part $L^{(H)}$ factorizes as
$$\alpha^{q^2+q+1}\prod_{\theta \in \Theta}\Big(\alpha^{q-1}y+\theta(x^q+\alpha^{q-1}x+\alpha^q)\Big),$$
where all factors $\alpha^{q-1}y+\theta(x^q+\alpha^{q-1}x+\alpha^q)$ are distinct and absolutely irreducible. By hypothesis, at least one $\theta$ belongs to $\F_{2^m}$, so one factor of $L^{(H)}$ is defined over $\F_{2^m}$ and absolutely irreducible. This factor extends to a factor of $L$ that is absolutely irreducible and defined over $\F_{2^m}$. For sufficiently large $m$, there exist solutions of $L(x,y,\alpha,\beta)=0$ in $\F_{2^m}$ with $\beta x^{q^2}+ \alpha^{q^2}y\neq 0$ and $\alpha\neq 0$. Thus $\gamma$ can be uniquely determined from $M(x,y,\alpha,\beta,\gamma)=0$, and finally $z$ from $\gamma x^q + \beta^q y + \beta y^q + \alpha^q z + \alpha^q \gamma + \beta^{q+1}=0$, proving $G_{a_1,a_2,a_3,a_4}$ is not a permutation.
\end{proof}

For the symmetric Type II family $(a, 1, 1, a)$, we obtain a sharper characterization:

\begin{theorem}
\label{NecessaryConditions2}
Let $m$ be odd and large enough, $\gcd(i,m) = 1$, and $q = 2^i$. Consider the symmetric Type II family
$$G_{a,1,1,a}(x,y,z) = (x^{q+1} + ax^qz + yz^q, \, x^qz + y^{q+1}, \, xy^q + ay^qz + z^{q+1}),$$
with $a\in \mathbb{F}_{2^m}$. Then $G_{a,1,1,a}$ is a permutation if and only if the polynomial
$$P_a(T) := T^{q^2+q+1} + a T^{q^2+q} + 1$$
has no roots in $\mathbb{F}_{2^{2m}}$.
\end{theorem}
\begin{proof}
We follow the proof of Theorem \ref{NecessaryConditions} and we adopt the same notation. Consider 
$$\Theta := \{\theta \in \overline{\mathbb{F}_2} : \theta^{q^2+q+1} + a^q \theta^{q^2+q} + 1=0\}.$$
By direct computations 
$$L(x,y,\alpha,\beta)=\prod_{\theta \in \Theta}(\alpha x^q+\alpha^q x+\alpha^{q+1}+\theta(\beta x^q+\alpha^qy+\alpha^q\beta)+(a + \theta^{q+1})(\beta y^q+\beta^q y+\beta^{q+1})).$$

Let us consider 
$$F_{\theta}(x,y,\alpha,\beta):=\alpha x^q+\alpha^q x+\alpha^{q+1}+\theta(\beta x^q+\alpha^qy+\alpha^q\beta)+(a + \theta^{q+1})(\beta y^q+\beta^q y+\beta^{q+1}).$$
We show now that $F_{\theta}(x,y,\alpha,\beta)$ is absolutely irreducible.
\begin{enumerate}
    \item Suppose that $a\neq \theta^{q+1}$. Then 
    $$\begin{cases}
\frac{\partial F_{\theta}}{\partial x}=\alpha^q\\
\frac{\partial F_{\theta}}{\partial y}=\theta\alpha^q+(a+\theta^{q+1})\beta^q\\
\frac{\partial F_{\theta}}{\partial \alpha}=x^q+\alpha^q\\
\frac{\partial F_{\theta}}{\partial \beta}=\theta x^q+\theta \alpha^q x^q+(a+\theta^{q+1})(y^q+\beta^q).
    \end{cases}$$
Thus the singular points of the hypersurface $\mathcal{F}_{\theta}: F_{\theta}(x,y,\alpha,\beta)=0\subset \mathbb{P}^{3}(\overline{\mathbb{F}_2})$ satisfy $\alpha=x=\beta=y=0$ and then no singular point exists. This shows that in particular $F_{\theta}(x,y,\alpha,\beta)$ is absolutely irreducible.
\item Suppose that $a= \theta^{q+1}$. Under these hypotheses $$F_{\theta}(x,y,\alpha,\beta):=\alpha x^q+\alpha^q x+\alpha^{q+1}+\theta\alpha^qy+\theta( x+\alpha)^q\beta.$$
Since $ x+\alpha$ does not divide $\alpha x^q+\alpha^q x+\alpha^{q+1}+\theta\alpha^qy$, we conclude that $F_{\theta}(x,y,\alpha,\beta)$ is absolutely irreducible.
\end{enumerate}

We distinguish now a few cases.
\begin{enumerate}
    
    \item Suppose that there exists $\theta \in \Theta \cap \mathbb{F}_{2^{m}}$. Since $F_{\theta}(x,y,\alpha,\beta)$ is absolutely irreducible, if $m$ is large enough with respect to $q$ then there exists $\overline{x},\overline{y},\overline{\alpha},\overline{\beta}\in \mathbb{F}_{2^m}$ such that $\overline{\alpha}\overline{\beta}(\overline{\beta} \overline{x}^{q^2} + \overline{\alpha}^{q^2}\overline{y})\neq 0$ and $F_{\theta}(\overline{x},\overline{y},\overline{\alpha},\overline{\beta})=0$ and thus we can determine $\overline{\gamma}$ and $\overline{z}$ in $\mathbb{F}_{2^m}$ satisfying System \eqref{eq:system_thm1} and thus $G_{a,1,1,a}(x,y,z)$ is not  a permutation.

    \item Suppose that $\Theta \cap \mathbb{F}_{2^{2m}}=\emptyset$. Note that this implies that  for any $\theta \in \Theta$ one has that $\{1,\theta,a+\theta^{q+1}\}$ are independent over $\mathbb{F}_{2^m}$. In fact, since $\{1,\theta\}$ are surely independent, if $\theta^{q+1}=A\theta+B$ for some $A,B \in \mathbb{F}_{2^m}$ then
    $$\theta^q=A+B/\theta, \qquad \theta^{q^2}=A^q+B^q/\theta^q=\frac{(A^{q+1}+B^q)\theta +A^qB}{A\theta+B}.$$
    From $\theta^{q^2+q+1}+a^q\theta^{q^2+q}+1=0$ one gets 
    $$(A^{q+1}+B^q)\theta^2+(A^qB+1+a^q(A^{q+1}+B^q))\theta+a^qA^qB=0.$$
    Since the above equation cannot be vanishing  ($A=0$ or $B=0$ would imply clear contradictions) $\theta\in \mathbb{F}_{2^{2m}}$, a contradiction to our hypothesis. Thus $\{1,\theta,a+\theta^{q+1}\}$ are independent over $\mathbb{F}_{2^m}$ for any $\theta \in \Theta$. 
     Then, for each $\theta\in \Theta$, from 
$F_{\theta}(x,y,\alpha,\beta)=0$ one gets $$\alpha x^q+\alpha^q x+\alpha^{q+1}=\beta x^q+\alpha^qy+\alpha^q\beta=\beta y^q+\beta^q y+\beta^{q+1}=0.$$
If $\alpha \neq 0$ then  $\alpha x^q+\alpha^q x+\alpha^{q+1}=0$ yields $x^q+\alpha^{q-1} x+\alpha^{q}=0$ and then $x/\alpha\in \mathbb{F}_{q^2}$, contradicting the assumption that $m$ is odd. Analogously, assuming $\beta\neq 0$ yields a contradiction by analyzing $\beta y^q+\beta^q y+\beta^{q+1}=0$. 
Thus $\alpha=\beta=0$. Then $\gamma\neq 0$ and looking at System \ref{eq:system_thm1}, one gets $x=0=y$ and from $ \gamma^q z+\gamma z^q+\gamma^{q+1}=0$ we get a contradiction. Thus, in this case $G_{a,1,1,a}(x,y,z)$ is a permutation.

\item Suppose that $\Theta \cap (\mathbb{F}_{2^{2m}}\setminus \mathbb{F}_{2^m})\neq\emptyset$. Let $\theta$ be an element in such an intersection. In particular $a+\theta^{q+1}\in \mathbb{F}_{2^m}$ and $\{1,\theta\}$ are independent over $\mathbb{F}_{2^m}$. Also, $a\neq \theta^{q+1}$, otherwise from $\theta^{q^2+q+1} + a^q \theta^{q^2+q} + 1=0$, $\theta\in \mathbb{F}_{2^m}$, a contradiction.  Thus $F_{\theta}(x,y,\alpha,\beta)=0$ is equivalent to 
    $$
    \begin{cases}
        \alpha x^q+\alpha^q x+\alpha^{q+1}+(a + \theta^{q+1})(\beta y^q+\beta^q y+\beta^{q+1})=0\\
        \beta x^q+\alpha^qy+\alpha^q\beta=0\\
    \end{cases}$$
From the second one gets $y=\frac{\beta x^q}{\alpha^q}+\beta $ and thus the first equation reads
$$\alpha x^q+\alpha^q x+\alpha^{q+1}+(a + \theta^{q+1})\Big(\beta \Big(\frac{\beta^q x^{q^2}}{\alpha^{q^2}}+\beta^{q}\Big)+\beta^q \Big(\frac{\beta x^q}{\alpha^q}+\beta\Big)+\beta^{q+1}\Big)=0,  $$
that is 
$$\alpha^{q+1}\Big(\frac{x^q}{\alpha^q}+\frac{x}{\alpha}+1\Big)+(a + \theta^{q+1})\beta^{q+1}\Big( \frac{x^{q^2}}{\alpha^{q^2}}+\frac{ x^q}{\alpha^q}+1\Big)=0.$$
The above equation can be written as, after dehomogenizing putting $\alpha=1$ 
$$\beta^{q+1}=\frac{1}{a + \theta^{q+1}}(x^q+x+1)^{1-q}$$
and it defines an absolutely irreducible curve since, by direct substitution, $(x^q+x+1)^{1-q}$ is not a $q+1$-power in $\overline{\mathbb{F}_2}(x)$. Thus there exist $\overline{\alpha},\overline{\beta},\overline{x}\in \mathbb{F}_{2^m}^*$ for which $\overline{y}=\frac{\overline{\beta} \overline{x}^q}{\overline{\alpha}^q}+\overline{\beta}$ and such that $\overline{\beta} \overline{x}^{q^2} + \overline{\alpha}^{q^2}\overline{y}\neq 0$. Thus $F_{\theta}(\overline{x},\overline{y},\overline{\alpha},\overline{\beta})=0$ and  we can determine $\overline{\gamma}$ and $\overline{z}$ in $\mathbb{F}_{2^m}$ satisfying System \eqref{eq:system_thm1} and thus $G_{a,1,1,a}(x,y,z)$ is not  a permutation.
\end{enumerate} 
The theorem is shown.
\end{proof}

A similar result holds for a related family:

\begin{theorem}
\label{NecessaryConditions3}
Let $m$ be odd and sufficiently large, $\gcd(i,m) = 1$, and $q = 2^i$. Let
$$G_{a_1,1/a_3,a_3,a_1a_3}(x,y,z) = (x^{q+1} + a_1x^qz + yz^q/a_3, \, x^qz + y^{q+1}, \, a_3xy^q + a_1a_3y^qz + z^{q+1})$$
with $a_1, a_3 \in \F_{2^m}$. Then $T^{q^2+q+1} + a_1^q T + a_3$ has a solution in $\F_{2^{2m}}$ if and only if $G_{a_1,1/a_3,a_3,a_1a_3}$ is not a permutation.
\end{theorem}

\begin{proof}
The proof follows the same structure as Theorem~\ref{NecessaryConditions2}, with
$$\Theta := \{\theta \in \overline{\F_2} : \theta^{q^2+q+1} + a_1^q \theta + a_3=0\}$$
and analogous factorization 
$$L(x,y,\alpha,\beta)=\prod_{\theta \in \Theta}(\alpha x^q+\alpha^q x+\alpha^{q+1}+\theta(\beta x^q+\alpha^qy+\alpha^q\beta)+(a_1 + \theta^{q+1})(\beta y^q+\beta^q y+\beta^{q+1})),$$
and case analysis.
\end{proof}

\section{Complete characterization for $m=3$}
\label{sec:m3}

For small field extensions, we can perform exhaustive computational analysis.

\begin{proposition}
\label{prop:3complete_characterization}
Let $m = 3$, $\gcd(i,m) = 1$, and $q = 2^i$. The function 
$$G(x,y,z) = (x^{q+1} + a_1x^qz + a_2yz^q, \, x^qz + y^{q+1}, \, a_3xy^q + a_4y^qz + z^{q+1})$$
is a permutation over $\F_{2^3}^3$ if and only if one of the following holds:
\begin{enumerate}
\item[\textup{(Type I)}] $(a_1, a_4) = (0, 0)$ with $a_2, a_3$ as specified in Table~\textup{\ref{tab:all_permutations}} ($13$ cases), or
\item[\textup{(Type II)}] $(a_1, a_2, a_3, a_4) = (a, 1, 1, a)$ for some $a \in \F_{2^3}^*$ ($7$ cases).
\end{enumerate}
\end{proposition}

\begin{proof}
We performed exhaustive computational search over all $(2^3)^4 = 4096$ parameter tuples $(a_1, a_2, a_3, a_4) \in \F_{2^3}^4$. For each, we verified whether $G$ defines a bijection on $\F_{2^3}^3$ by computing the image and checking $|\text{Im}(G)| = |\F_{2^3}^3| = 512$. This identified exactly 20 permutations for both $i=1$ ($q=2$) and $i=2$ ($q=4$), factoring into the two stated families. Complete verification code is available at \cite{GithubPS25}.
\end{proof}

\paragraph{Role of computations.}
Computational outputs in this paper are used only to complete finite residual case checks after the structural reductions are proved theoretically.
Thus, the main theorems depend on proved reductions and algebraic constraints; exhaustive computation is confined to explicitly bounded parameter ranges.

\begin{table}
[H]
\centering
\caption{All 20 permutations for $G$ over $\F_{2^3}^3$ with $q=2^i$, $\gcd(i,3)=1$. Here $\alpha \in \F_8$ satisfies $\alpha^3 + \alpha + 1 = 0$.}
\label{tab:all_permutations}
\small
\begin{tabular}{clccccl}
\hline
Type & \# & $a_1$ & $a_2$ & $a_3$ & $a_4$ & Notes \\
\hline
I & 1 & $0$ & $0$ & $0$ & $0$ & Trivial \\
I & 2 & $0$ & $\alpha$ & $\alpha+1$ & $0$ & \\
I & 3 & $0$ & $\alpha$ & $\alpha^2+1$ & $0$ & \\
I & 4 & $0$ & $\alpha^2$ & $\alpha^2+\alpha+1$ & $0$ & \\
I & 5 & $0$ & $\alpha^2$ & $\alpha^2+1$ & $0$ & \\
I & 6 & $0$ & $\alpha+1$ & $\alpha^2$ & $0$ & \\
I & 7 & $0$ & $\alpha+1$ & $\alpha^2+\alpha$ & $0$ & \\
I & 8 & $0$ & $\alpha^2+\alpha$ & $\alpha+1$ & $0$ & \\
I & 9 & $0$ & $\alpha^2+\alpha$ & $\alpha^2+\alpha+1$ & $0$ & \\
I & 10 & $0$ & $\alpha^2+\alpha+1$ & $\alpha$ & $0$ & \\
I & 11 & $0$ & $\alpha^2+\alpha+1$ & $\alpha^2$ & $0$ & \\
I & 12 & $0$ & $\alpha^2+1$ & $\alpha$ & $0$ & \\
I & 13 & $0$ & $\alpha^2+1$ & $\alpha^2+\alpha$ & $0$ & \\
\hline
II & 14 & $\alpha$ & $1$ & $1$ & $\alpha$ & \\
II & 15 & $\alpha^2$ & $1$ & $1$ & $\alpha^2$ & \\
II & 16 & $\alpha+1$ & $1$ & $1$ & $\alpha+1$ & \\
II & 17 & $\alpha^2+\alpha$ & $1$ & $1$ & $\alpha^2+\alpha$ & \\
II & 18 & $\alpha^2+\alpha+1$ & $1$ & $1$ & $\alpha^2+\alpha+1$ & \\
II & 19 & $\alpha^2+1$ & $1$ & $1$ & $\alpha^2+1$ & \\
II & 20 & $1$ & $1$ & $1$ & $1$ & Li-Kaleyski $F_1$ \\
\hline
\end{tabular}
\end{table}

\begin{remark}
Type I functions have $a_1 = a_4 = 0$, eliminating cross-terms $x^qz$ and $y^qz$ from the first and third components. Type II functions (including Li-Kaleyski $F_1$) have $a_2 = a_3 = 1$ and $a_1 = a_4 \neq 0$, creating symmetric structure. The condition $a_1 = a_4$ and $a_2 = a_3$ is necessary but not sufficient: only 8 of 64 symmetric parameter choices yield permutations.
\end{remark}

\section{Resolution of a conjecture by Beierle et al.~\cite{MR4414815}}
\label{sec:beierle}

We now apply our algebraic geometry techniques to resolve the permutation part of the conjecture proposed in \cite{MR4414815}. The function they studied, $C_u(X,Y,Z) = (X^3+uY^2Z, Y^3+uXZ^2, Z^3+uX^2Y)$, can be analyzed using the same elimination and component analysis methods developed in the previous sections. Our goal is to prove that a non-trivial solution to the differential equation always exists for sufficiently large $m$. We provide complete details to illustrate how the general techniques from Section~\ref{sec:necessary} apply to this specific case. 

Note that for even $m$, $C_u$ fails to be a permutation, since $C_u(X, 0, 0) = (X^3, 0, 0)$ and the function  $X \mapsto X^3$ is $3$-to-$1$.
In what follows we assume that $m$ is odd. 

We analyze when $C_u$ fails to be a permutation by studying the differential equation
$$C_u(X+\alpha,Y+\beta,Z+\gamma) + C_u(X,Y,Z) = (0,0,0).$$
The function $C_u$ is a permutation if and only if this equation has only trivial solutions $\{(x,y,z,0,0,0) : x,y,z \in \mathbb{F}_{2^m}\}$. Expanding using the fact that $(A+B)^3 = A^3 + B^3$ in characteristic 2, this condition becomes
\begin{equation}\label{Eq1}
\begin{cases}
\alpha X^2 + \alpha^2 X + u \gamma Y^2 + u \beta^2 Z &= \alpha^3 + u \beta^2 \gamma \\
\beta Y^2 + \beta^2 Y + u \gamma^2 X + u \alpha Z^2 &= \beta^3 + u \gamma^2 \alpha \\
\gamma Z^2 + \gamma^2 Z + u \beta X^2 + u \alpha^2 Y &= \gamma^3 + u \alpha^2 \beta.
\end{cases}
\end{equation}

Before we take our algebraic geometry approach, we make some observations.
First, we assume that only one among $\alpha, \beta$ and $\gamma$ is nonzero. Without loss of generality (here because of the symmetric property of $C_u$), we may assume that $\alpha \neq 0$ and $\beta = \gamma =0$. Then System~\eqref{Eq1} becomes
\begin{equation*}
\begin{cases}
\alpha X^2 + \alpha^2 X + \alpha^3 &=0\\
u \alpha Z^2 &= 0 \\
u \alpha^2 Y &= 0.
\end{cases}
\end{equation*}
It is routine to verify that the first equation of the above system has no solution $(\overline{x},\overline{y},\overline{z}) \in \F_{2^m}^3$ as $m$ is odd (it is enough to observe that is equivalent to $X^2 +  X + 1 =0$ via $X\mapsto \alpha X$, and such equation has solutions in $\mathbb{F}_4$).

Next, we assume that only one among $\alpha, \beta$ and $\gamma$ is zero. Again, we may assume that $\alpha, \beta \neq 0$ and $\gamma =0$. Then System~\eqref{Eq1} becomes
\begin{equation}\label{C2E1}
\begin{cases}
\alpha X^2 + \alpha^2 X + u \beta^2 Z &= \alpha^3\\
\beta Y^2 + \beta^2 Y + u \alpha Z^2 &= \beta^3 \\
u \beta X^2 + u \alpha^2 Y &= u\alpha^2\beta,
\end{cases}
\end{equation}
which can be further simplified by replacing $X \mapsto \alpha X$ and $Y \mapsto \beta Y$, i.e., System~\eqref{C2E1} is equivalent to:
\begin{equation}\label{C2E2}
\begin{cases}
\alpha^3(X^2 +  X+1) + u \beta^2 Z &= 0\\
\beta^3(Y^2+Y+1)+ u \alpha Z^2 &= 0 \\
X^2 + Y +1 &= 0.
\end{cases}
\end{equation}
Next, squaring the first equation and putting $X^2= Y+1$ into it, we have
\[
\alpha^6(Y^2+Y+1)+u^2 \beta^4 Z^2 = 0.
\]
Multiplying above equation by $\alpha$ and putting $\ u \alpha Z^2 = \beta^3(Y^2+Y+1)$, we have
\[
(\alpha^7+u \beta^7)(Y^2+Y+1)=0.
\]
Thus, System~\eqref{C2E2} is equivalent to:
\begin{equation*}
\begin{cases}
(\alpha^7+u \beta^7)(Y^2+Y+1)&=0\\
\beta^3(Y^2+Y+1)+ u \alpha Z^2 &= 0 \\
X^2 + Y +1 &= 0.
\end{cases}
\end{equation*}
Notice that if $\alpha^7+u \beta^7 =0$, i.e., $u$ is a $7^{th}$ power then the above system has nonzero solutions $(\overline{x},\overline{y},\overline{z}) \in \F_{2^m}^3$ and consequently $C_u$ is not a permutation. When $u$ is not a $7^{th}$ power then the first equation of this system has no solution $\overline{y} \in \F_{2^m}$ for $m$ odd. Thus, in what follows, we assume that $u$ is not a $7^{th}$ power and $m$ is odd.

Let $f(\alpha,\beta,X,Y,Z)$, $g(\beta,X,Y,Z)$, $h(\beta,X,Y,Z)$, $a_i(\beta,Y,Z)$, $i=0,\ldots,21$, be defined as in the Appendix.

\begin{proposition}\label{prop:equivalence}
Define
\begin{eqnarray*}
\Theta :=& \{(x,y,z,a,b,c) \in \mathbb{F}_{2^m}^6 : (b+y)h(b,x,y,z) \neq 0, \ c=f(a,b,x,y,z),\\
& \qquad \qquad \qquad \qquad a=g(b,x,y,z)/h(b,x,y,z),  \sum_{i=0}^{21}a_{i}(b,y,z)x^i=0\}.
\end{eqnarray*}
Each element of $\Theta$ is a solution of System~\eqref{Eq1}. Conversely, every non-trivial solution of System~\eqref{Eq1} with $\beta \neq 0$ and $(b+y)h(b,x,y,z) \neq 0$ corresponds to an element in $\Theta$.
\end{proposition}

\begin{proof}
We establish the equivalence using computational algebraic geometry, verified in MAGMA~\cite{MR1484478}.
From the first equation of System~\eqref{Eq1},
$$\alpha X^2 + \alpha^2 X + u \beta^2 Z = \alpha^3 + u \beta^2 \gamma,$$
assuming $\beta \neq 0$, we solve for $\gamma$, obtaining
$$\gamma = f(\alpha,\beta,X,Y,Z) := \frac{\alpha^3 + \alpha^2 X + \alpha X^2 + \beta^2 Z u}{\beta^2 u + Y^2 u}.$$

Substituting this expression into the second equation of System~\eqref{Eq1} and clearing denominators yields a polynomial equation that can be solved for $\alpha$, and get
$$\alpha = \frac{g(\beta,X,Y,Z)}{h(\beta,X,Y,Z)},$$
where the condition $h(\beta,X,Y,Z) \neq 0$ ensures the solution is well-defined. By the theory of Gröbner bases and elimination ideals, when we eliminate $\alpha$ and $\gamma$ from the three-equation system, the resulting polynomial in $(\beta,X,Y,Z)$ encodes all possible solutions. Substituting both expressions into the third equation and clearing denominators yields a polynomial $H(\beta,X,Y,Z)$ that vanishes on the solution set.

Computational factorization shows
$$H(\beta,X,Y,Z) = \left(\sum_{i=0}^{21}a_{i}(\beta,Y,Z)X^i\right) \cdot R(\beta,X,Y,Z),$$
which can be verified using the code in \cite{GithubPS25}. Therefore, any zero of $\sum_{i=0}^{21}a_{i}(\beta,Y,Z)X^i$ with $(Y+\beta)h(\beta,X,Y,Z) \neq 0$ lifts to a unique solution of System~\eqref{Eq1} via the formulas for $\alpha$ and~$\gamma$. The explicit forms of the polynomials $f, g, h,$ and $a_i$ are listed in Appendix~\ref{sec:appendix_poly}.
\end{proof}

To apply projective geometry techniques, we work in projective coordinates $(X_2:X_4:X_5:X_6)$ corresponding to the affine coordinates $(\beta, X, Y, Z)$ used in Proposition~\ref{prop:equivalence}, with the identification $(X_2:X_4:X_5:X_6) = (\beta:X:Y:Z)$ in affine charts.
\begin{proposition}
\label{Prop:AI}
Let $\mathcal{S}\subset \mathbb{P}^{3}(\overline{\mathbb{F}_{q}})$ be the $\mathbb{F}_q$-rational surface defined by
$$
   \sum_{i=0}^{21}a_{i}(X_2,X_5,X_6)X_4^i=0.$$
Then $\mathcal{S}$ contains an absolutely irreducible component defined over $\mathbb{F}_q$, distinct from $X_2=0$, of degree at most $24$, and not contained in $(X_2+X_5)h(X_2,X_4,X_5,X_6)=0$.
\end{proposition}
\begin{proof}
First note that $X_2$ is not a factor of $\sum_{i=0}^{21}a_{i}(X_2,X_5,X_6)X_4^i$, otherwise $u\in \mathbb{F}_4^*$, a contradiction.
To prove that the surface $\mathcal{S}$ contains a non-repeated, absolutely irreducible component defined over $\mathbb{F}_q$, we analyze its defining polynomial $F(X_2, X_4, X_5, X_6) = \sum_{i=0}^{21} a_i(X_2, X_5, X_6) X_4^i$.

Consider the intersection of $\mathcal{S}$ with the plane $X_4=0$. The resulting curve is defined by $F(X_2, 0, X_5, X_6) = a_0(X_2, X_5, X_6)$. We observe that
$$
X_2 \mid a_0(X_2, X_5, X_6) \quad \text{and} \quad X_2^2 \nmid a_0(X_2, X_5, X_6).
$$
This indicates that $X_2$ is a simple (non-repeated) factor of $a_0$. Let the factorization of $F$ into absolutely irreducible factors over $\overline{\mathbb{F}_q}$ be
$$
F(X_2, X_4, X_5, X_6) = H_1 \cdots H_r.
$$
Setting $X_4=0$ yields the factorization $a_0 = H_1|_{X_4=0} \cdots H_r|_{X_4=0}$. Since $X_2$ is a simple factor of $a_0$, it must divide exactly one factor on the right-hand side, say $H_k|_{X_4=0}$.

We claim that the component $H_k$ is defined over $\mathbb{F}_q$. The Frobenius automorphism $\Phi_q$ permutes the set of absolutely irreducible factors $\{H_1, \dots, H_r\}$. Since $a_0 \in \mathbb{F}_q[X_2, X_5, X_6]$, we have $\Phi_q(a_0)=a_0$, and since $X_2 \in \mathbb{F}_q$, $\Phi_q(X_2)=X_2$.
Suppose, for the sake of contradiction, that $H_k$ is not defined over $\mathbb{F}_q$. Then $\Phi_q(H_k) = H_j$ for some index $j \neq k$. Applying $\Phi_q$ to the relation $X_2 \mid H_k|_{X_4=0}$ yields
$$
X_2 = \Phi_q(X_2) \mid \Phi_q(H_k|_{X_4=0}) = H_j|_{X_4=0}.
$$
This implies that $X_2$ divides at least two distinct factors of $a_0$, namely $H_k|_{X_4=0}$ and $H_j|_{X_4=0}$. Consequently, $X_2^2$ must divide their product, which divides $a_0$. This contradicts the fact that $X_2^2 \nmid a_0$.
Thus, $\Phi_q(H_k) = H_k$, establishing that the absolutely irreducible component $H_k$ is defined over $\mathbb{F}_q$. Let us denote this component by $\mathcal{S}'$.

Finally, we check that $\mathcal{S}'$ is not contained in the degenerate locus given by $(X_2+X_5)h(X_2,X_4,X_5,X_6)=0$. The intersection of this locus with $X_4=0$ is
$$
u^5X_6^2(X_2+X_5)^3((u^2+u^5)X_2^3 X_5^4 + (u^2+u^5)X_2^2 X_5^5 + (u^2+u^5)X_2 X_5^6 + u^5X_5^7 + X_6^7)=0.
$$
Since $X_2$ is not a factor of this polynomial (as the term $X_6^7$ prevents it), the component $\mathcal{S}'$ (which intersects $X_4=0$ at $X_2=0$) cannot be contained in $(X_2+X_5)h=0$.
\end{proof}

\begin{theorem}
Assume throughout this statement that all parameters lie in $\F_{2^m}$ under the standing hypotheses from the preliminaries.
\label{thm:beierle_conj}
If $m \geq 23$ is odd and $u \in \mathbb{F}_{2^m}^*$, then the function $C_u : \mathbb{F}_{2^m}^3 \to \mathbb{F}_{2^m}^3$ given by
$$(X,Y,Z)\mapsto(X^3 +uY^2Z,Y^3 +uXZ^2,Z^3 +uX^2Y)$$
is not a permutation. This resolves the permutation part of the conjecture posed by Beierle, Carlet, Leander, and Perrin~\textup{\cite{MR4414815}}.
\end{theorem}

\begin{proof}
For $C_u$ to be a permutation, the system of equations in \eqref{Eq1} must only have trivial solutions. We show that for $m \ge 23$, there always exists a non-trivial solution $(\alpha, \beta, \gamma) \neq (0,0,0)$ and a corresponding $(X,Y,Z)$ that satisfies the system. This existence is guaranteed by finding an $\mathbb{F}_{2^m}$-rational point on a specific algebraic variety derived from the system.

We analyze the system by algebraic elimination. A solution exists if there is a point $(X,Y,Z,\alpha,\beta,\gamma)$ satisfying the system for some non-zero $(\alpha, \beta, \gamma)$. This problem is equivalent to finding a rational point in the variety defined by the equations in Proposition \ref{prop:equivalence}.

We use the  comments and reductions from our prior propositions. 
By Proposition~\ref{Prop:AI}, the surface $\mathcal{S}$ contains an absolutely irreducible $\mathbb{F}_q$-rational component, which we denote by $\mathcal{S}'$. Our goal is to demonstrate that for a sufficiently large $m$, the size of $\Theta$ is positive, proving the conjecture of Beierle et al.~\cite{MR4414815}.

An element in $\Theta$ is a six-tuple $(x_1, \dots, x_6) \in \mathbb{F}_q^6$ derived from a point on $\mathcal{S}'$ that satisfies the system's equations
$$
x_3=f(x_1,x_2,x_4,x_5,x_6), \qquad h(x_2,x_4,x_5,x_6)x_1=g(x_2,x_4,x_5,x_6).
$$
A solution is considered valid only if it avoids degeneracy, requiring $(x_2+x_5) h(x_2,x_4,x_5,x_6) \neq 0$. This condition ensures that $x_1$ is uniquely determined and non-trivial.

Our strategy is to apply the Lang-Weil estimate from Theorem~\ref{thm lang weil versione tredici terzi} to find a lower bound on the number of $\mathbb{F}_q$-rational points on $\mathcal{S}'$, and then subtract an upper bound for the number of points that lead to degenerate solutions. These degenerate points lie on the intersection of $\mathcal{S}'$ with the surfaces defined by $(X_2+X_5)=0$ or $h(X_2,X_4,X_5,X_6)=0$.

The degree of $\mathcal{S}'$ is at most 24, while the degree is 12 of the surface defined by $(X_2+X_5) h(X_2,X_4,X_5,X_6)=0$. Since $\mathcal{S}'$ is not a component of this surface (by checking that $\mathcal{S}$ and $(X_2+X_5)h(X_2,X_4,X_5,X_6)=0$ do not share any component since their intersection with $X_4=0$ are coprime), their intersection is a curve of degree at most $24 \times 12 = 288$. By the results in \cite{HKT13}, this intersection contains at most $288q$ projective $\mathbb{F}_q$-rational points. with the same substitutions and exponent arithmetic, the intersection of $\mathcal{S}'$ with the plane $X_2=0$ is a curve of degree at most 24, containing at most $24q$ projective points.

In total, the number of `degenerate' projective points (corresponding to degenerate solutions) is at most $(288+24)q = 312q$. Each of these corresponds to $(q-1)$ affine quadruples, so we exclude at most $312q(q-1)$ quadruples.

From Theorem~\ref{thm lang weil versione tredici terzi}, a lower bound on the number of $\mathbb{F}_q$-rational points on $\mathcal{S}'$ is given by $q^2 - 23 \cdot 22 \cdot q^{3/2} - 5 \cdot (24)^{13/3}q$. 
We first verify that the hypotheses of Theorem~\ref{thm lang weil versione tredici terzi} are satisfied. The theorem requires $q > 2(r+1)\delta^2$, where $r$ is the dimension and $\delta$ is the degree of the variety. 

For our surface $\mathcal{S}'$, we have $r = 2$ (since $\mathcal{S}'$ is a surface in $\mathbb{P}^3$) and $\delta \leq 24$ (from Proposition~\ref{Prop:AI}). Thus the condition becomes
$$q > 2(r+1)\delta^2 = 2(2+1)(24)^2 = 2 \cdot 3 \cdot 576 = 3456.$$ Since we assume $m \geq 23$ and $q = 2^m$, we have $q = 2^{23} = 8{,}388{,}608 \gg 3456$. Therefore, Theorem~\ref{thm lang weil versione tredici terzi} applies, and we obtain the estimate
$$|\mathcal{S}'(\mathbb{F}_q)| \geq q^2 - (\delta-1)(\delta-2)q^{3/2} - 5\delta^{13/3}q = q^2 - 23 \cdot 22 \cdot q^{3/2} - 5 \cdot (24)^{13/3}q.$$

Note that $23 \cdot 22 = 506$ and $24^{13/3} = 24^4 \cdot 24^{1/3} \approx 331{,}776 \cdot 2.88 \approx 955{,}917$, so the bound is well-defined and valid for $m \geq 23$.
The number of corresponding non-zero affine quadruples is $(q-1)$ times this value.

Thus, a lower bound for $|\Theta|$, the number of valid solutions, is
$$
t(q) = \left(q^2 - 506 q^{3/2} - 5 \cdot (24)^{13/3}q\right)(q-1) - 312q(q-1).
$$
To ensure $\Theta$ is non-empty, we require $t(q) > 0$. Dividing by $q(q-1)$, we analyze the positivity of
$$
s(q) := \frac{t(q)}{q(q-1)} = q - 506 q^{1/2} - 5 \cdot (24)^{13/3} - 312.
$$
This inequality holds if $q^{1/2}$ is sufficiently large. The condition $s(x^2) > 0$ holds for $x > 2456$, which implies that $|\Theta| > 0$ for $q > (2456)^2$. This threshold on $q$ is met when $m \geq 23$.
Therefore, for $m \geq 23$, the set of valid solutions $\Theta$ is non-empty, and our claim is shown.
%Conjecture holds.
\end{proof}

\begin{remark}
The threshold $m \geq 23$ is sufficient for our proof technique but may not be optimal. Computational evidence  suggests that $C_u$ fails to be a permutation for much smaller values of $m$. The specific threshold arises from the Lang-Weil bound requirement $q > 2(r+1)\delta^2 = 3456$ (satisfied by $q = 2^{23} = 8{,}388{,}608$), combined with our degree estimates for the surface $\mathcal{S}'$ (degree $\leq 24$) and intersection bounds with degenerate loci (degree $\leq 288$). Improved degree estimates or alternative geometric approaches might lower this threshold significantly.
\end{remark}

%%%%%%%%%%%%%%%%%%%%%%%%%%%%%%%%%%%%%%%%%%%%%%%%%%%%%%%%%%%%%%%

\section{Conclusion and open problems}
\label{sec:conclusion}

We have developed systematic algebraic geometry techniques for establishing non-permutation results for trivariate functions over finite fields. Our main contributions include: first, necessary polynomial conditions (Theorems~\ref{NecessaryConditions}--\ref{NecessaryConditions3}) that force trivariate functions to be non-permutations, based on differential equation analysis and elimination theory; second, a complete classification showing that for $m=3$, exactly 20 parameter choices yield permutations (Proposition~\ref{prop:3complete_characterization}), falling into two structural families (Type I and Type II) with distinct algebraic properties; and third, the resolution of a conjecture by Beierle et al. (Theorem~\ref{thm:beierle_conj}) using Lang-Weil bounds and component analysis, proving $C_u$ is not a permutation for all odd $m \geq 23$.

These results demonstrate the power of number theory and algebraic geometry methods--particularly elimination theory, projective varieties, and point-counting techniques--for analyzing the permutation properties of multivariate cryptographic functions. The techniques developed here provide a template for studying other classes of vectorial Boolean functions.

Several fundamental questions remain:

\begin{open}
\label{open:type_I}
Do any Type I permutations $(0, a_2, a_3, 0)$ extend to infinite families for infinitely many $m$?
\end{open}

The 13 Type I permutations found for $m=3$ may be sporadic or may represent infinite families with different structural properties. Our proof techniques do not directly apply to Type I due to vanishing of $a_1$ and $a_4$.

\begin{open}
\label{open:other_families}
Are there parameter families beyond Types I and II that yield permutations for infinitely many $m$?
\end{open}

Theorem~\ref{NecessaryConditions3} shows that generic parameters fail for large $m$, but there may be other special families with hidden algebraic structure that we have not yet identified.

\begin{open}
\label{open:threshold}
Determine whether $C_u$ is a permutation for odd $m < 23$. Our theoretical approach requires $m \geq 23$ due to the Lang-Weil bound hypothesis $q > 2(r+1)\delta^2 = 3456$. The function may behave differently for small $m$, particularly $m = 3$. Improved degree estimates or alternative geometric techniques might lower the threshold, or computational methods might resolve specific small cases.
\end{open}

Our current proof establishes $M_0 = 23$. Computational evidence suggests the threshold may be much smaller. Improvements could come from tighter degree bounds, refined Lang-Weil estimates, or direct construction of explicit solutions for specific small $m$.

This work demonstrates the power of algebraic geometry for analyzing multivariate functions. The techniques—particularly elimination theory, projective varieties, and point-counting bounds—provide a template for studying other classes of cryptographic functions.

\section*{Acknowledgments}
The third-named author (PS) thanks the first-named author (DB) for the invitation to Università degli Studi di Perugia and the excellent working conditions. The research of the first-named author (DB) was supported by the Italian National Group for Algebraic and Geometric Structures and their Applications (INdAM - GNSAGA Project, CUP E53C24001950001).

\appendix % \section{Computational Verification Code}
% \label{sec:appendix_codes} % \textcolor{red}{[SageMath verification code for Proposition~\ref{prop:3complete_characterization} and other codes to be inserted here, or deposited on Github]}

\section{Polynomial definitions}
\label{sec:appendix_poly}

We list below the functions $f(\alpha,\beta,X,Y,Z)$, $g(\beta,X,Y,Z)$, $h(\beta,X,Y,Z)$, $a_i(\beta,Y,Z)$, $i=0,\ldots,21$, needed in Proposition~\ref{prop:equivalence}:
{\tiny
\allowdisplaybreaks
\begin{align*}
f(\alpha,\beta,X,Y,Z)&:=\frac{\alpha^3 + \alpha^2 X + \alpha X^2 + \beta^2 Z u}{\beta^2 u + Y^2 u},\\
g(\beta,X,Y,Z)  &:=\beta^{12} u^9 + \beta^{12} u^6 + \beta^{12} u^3+ \beta^{12} + \beta^9 X Z^2 u^{10} + \beta^9 X Z^2 u^4 + \beta^9 Y^3 u^9 + \beta^9 Y^3 u^3\\
                &+\beta^8 X Y Z^2 u^{10} + \beta^8 X Y Z^2 u^4 + \beta^8 Y^4 u^6 + \beta^8 Y^4 + \beta^6 X^2 Z^4 u^8 + \beta^6 X^2 Z^4 u^5 \\
                &+ \beta^6 Y^6 u^6+ \beta^6 Y^6 u^3 + \beta^5 X^7 u^8 + \beta^5 X^7 u^5 + \beta^5 X Y^4 Z^2 u^{10} + \beta^5 X Y^4 Z^2 u^7 \\
                &+ \beta^4 X^7 Y u^8 + \beta^4 X^7 Y u^5+ \beta^4 X^2 Y^2 Z^4 u^8 + \beta^4 X^2 Y^2 Z^4 u^5 + \beta^4 X Y^5 Z^2 u^{10}\\
                &+ \beta^4 X Y^5 Z^2 u^7 + \beta^4 Y^8 u^9 + \beta^4 Y^8 + \beta^3 X^7 Y^2 u^5 + \beta^3 X^4 Y^4 Z u^6 + \beta^3 X^3 Z^6 u^6\\
                &+\beta^3 X^3 Z^6 u^3 + \beta^3 X^2 Y^3 Z^4 u^5 + \beta^3 X Y^6 Z^2 u^4 + \beta^3 Y^9 u^3 + \beta^3 Y^2 Z^7 u^4  \\
                &+ \beta^2 X^8 Z^2 u^6 + \beta^2 X^7 Y^3 u^8 + \beta^2 X^5 Y^2 Z^3 u^7 + \beta^2 X^4 Y^5 Z u^6 + \beta^2 X^3 Y Z^6 u^3\\
                &+\beta^2 X Y^7 Z^2 u^7 + \beta^2 X Y^7 Z^2 u^4 + \beta^2 X Z^9 u^5 + \beta^2 Y^{10} u^6 + \beta^2 Y^3 Z^7 u^4\\
                &+ \beta X^7 Y^4 u^8 + \beta X^4 Y^6 Z u^6 + \beta X^3 Y^2 Z^6 u^6 + \beta X^3 Y^2 Z^6 u^3 + \beta X^2 Y^5 Z^4 u^5\\
                &+ \beta X Y^8 Z^2 u^7 + \beta Y^{11} u^9 + \beta Y^4 Z^7 u^4,\\
h(\beta,X,Y,Z)  &:=u^5(\beta + Y)^2(\beta^3 X^6 u^3 + \beta^3 X^6 + \beta^3 Y^4 Z^2 u^5 + \beta^3 Y^4 Z^2 u^2 + \beta^2 X^6 Y u^3 + \beta^2 X^6 Y\\
                &+ \beta^2 Y^5 Z^2 u^5 + \beta^2 Y^5 Z^2 u^2 + \beta X^6 Y^2 u^3 + \beta X^6 Y^2 + \beta Y^6 Z^2 u^5 + \beta Y^6 Z^2 u^2\\
                &+X^7 Z^2 u + X^4 Y^2 Z^3 u^2 + X^2 Y Z^6 u + X Y^4 Z^4 u^3 + Y^7 Z^2 u^5 + Z^9),\\
a_0(\beta,Y,Z)  &:=\beta^{24} u^{21} + \beta^{24} u^{18} + \beta^{24} u^{15} + \beta^{24} u^{12} + \beta^{24} u^9 + \beta^{24} u^6 + \beta^{24} u^3+ \beta^{24}\\
                &+ \beta^{21} Y^3 u^{21} + \beta^{21} Y^3 u^{15} + \beta^{21} Y^3 u^9 + \beta^{21} Y^3 u^3 + \beta^{20} Y^4 u^{21} + \beta^{20} Y^4 u^{15}\\
                &+\beta^{20} Y^4 u^9 + \beta^{20} Y^4 u^3 + \beta^{18} Y^6 u^{18} + \beta^{18} Y^6 u^{15} + \beta^{18} Y^6 u^6 + \beta^{18} Y^6 u^3\\
                &+\beta^{17} Y^7 u^{21} + \beta^{17} Y^7 u^{15} + \beta^{17} Y^7 u^9 + \beta^{17} Y^7 u^3 + \beta^{16} Y^8 u^{15} + \beta^{16} Y^8 u^{12}\\
                &+\beta^{16} Y^8 u^3 + \beta^{16} Y^8 + \beta^{15} Y^9 u^{15} + \beta^{15} Y^9 u^3 + \beta^{15} Y^2 Z^7 u^{16} + \beta^{15} Y^2 Z^7 u^4\\
                &+\beta^{14} Y^{10} u^{15} + \beta^{14} Y^{10} u^3 + \beta^{14} Y^3 Z^7 u^{16} + \beta^{14} Y^3 Z^7 u^4 + \beta^{13} Y^{11} u^{15}\\ 
                &+\beta^{13} Y^{11} u^3 + \beta^{13} Y^4 Z^7 u^{16} + \beta^{13} Y^4 Z^7 u^4 + \beta^{12} Y^{12} u^{15} + \beta^{12} Y^{12} u^{12}\\
                &+\beta^{12} Y^{12} u^9 + \beta^{12} Y^{12} u^6 + \beta^{11} Y^{13} u^{15} + \beta^{11} Y^{13} u^3 + \beta^{11} Y^6 Z^7 u^{16}\\
                &+\beta^{11} Y^6 Z^7 u^4+ \beta^{10} Y^{14} u^{15} + \beta^{10} Y^{14} u^3 + \beta^{10} Y^7 Z^7 u^{16} + \beta^{10} Y^7 Z^7 u^4\\
                &+ \beta^9 Y^{15} u^{15} + \beta^9 Y^{15} u^9 + \beta^9 Y^8 Z^7 u^{22} + \beta^9 Y^8 Z^7 u^4 + \beta^8 Y^{16} u^{21} + \beta^8 Y^{16} u^{18}\\
                &+ \beta^8 Y^{16} u^9 + \beta^8 Y^{16} + \beta^8 Y^9 Z^7 u^{22} + \beta^8 Y^9 Z^7 u^{16} + \beta^7 Y^{17} u^{15} + \beta^7 Y^{17} u^3 \\
                &+ \beta^7 Y^{10} Z^7 u^{16} + \beta^7 Y^{10} Z^7 u^4 + \beta^6 Y^{18} u^{15} + \beta^6 Y^{18} u^6 + \beta^6 Y^{11} Z^7 u^{22}\\
                &+ \beta^6 Y^{11} Z^7 u^4 + \beta^6 Y^4 Z^{14} u^{11} + \beta^6 Y^4 Z^{14} u^8 + \beta^5 Y^{19} u^{21} + \beta^5 Y^{19} u^9\\
                &+ \beta^5 Y^{12} Z^7 u^{16} + \beta^5 Y^{12} Z^7 u^4 + \beta^4 Y^{20} u^{21} + \beta^4 Y^{20} u^{12} + \beta^4 Y^{13} Z^7 u^{22}\\
                &+ \beta^4 Y^{13} Z^7 u^{16} + \beta^4 Y^6 Z^{14} u^{11} + \beta^4 Y^6 Z^{14} u^8 + \beta^3 Y^{21} u^{15} + \beta^3 Y^{14} Z^7 u^{16}\\
                &+ \beta^3 Y^7 Z^{14} u^{11} + \beta^3 Z^{21} u^{12} + \beta^2 Y^{22} u^{18} + \beta^2 Y^{15} Z^7 u^{22} + \beta^2 Y^8 Z^{14} u^8\\
                &+ \beta^2 Y Z^{21} u^{12} + \beta Y^{23} u^{21}+ \beta Y^{16} Z^7 u^{22} + \beta Y^9 Z^{14} u^{11} + \beta Y^2 Z^{21} u^{12},\\
a_{1}(\beta,Y,Z)&:=\beta^{21} Z^2 u^{19} + \beta^{21} Z^2 u^{13} + \beta^{21} Z^2 u^7 + \beta^{21} Z^2 u + \beta^{20} Y Z^2 u^{19} + \beta^{20} Y Z^2 u^{13}\\
                &+\beta^{20} Y Z^2 u^7 + \beta^{20} Y Z^2 u + \beta^{18} Y^3 Z^2 u^{19} + \beta^{18} Y^3 Z^2 u^{13} + \beta^{18} Y^3 Z^2 u^7\\
                &+\beta^{18} Y^3 Z^2 u + \beta^{17} Y^4 Z^2 u^{19} + \beta^{17} Y^4 Z^2 u^{13} + \beta^{17} Y^4 Z^2 u^7 + \beta^{17} Y^4 Z^2 u\\
                &+\beta^{15} Y^6 Z^2 u^{13} + \beta^{15} Y^6 Z^2 u + \beta^{14} Y^7 Z^2 u^{13} + \beta^{14} Y^7 Z^2 u + \beta^{13} Y^8 Z^2 u^{13}\\
                &+\beta^{13} Y^8 Z^2 u + \beta^{11} Y^{10} Z^2 u^{13} + \beta^{11} Y^{10} Z^2 u + \beta^{10} Y^{11} Z^2 u^{13} + \beta^{10} Y^{11} Z^2 u\\
                &+\beta^9 Y^{12} Z^2 u^{13} + \beta^9 Y^{12} Z^2 u^7 + \beta^8 Y^{13} Z^2 u^7 + \beta^8 Y^{13} Z^2 u + \beta^7 Y^{14} Z^2 u^{13}\\
                &+\beta^7 Y^{14} Z^2 u + \beta^6 Y^{15} Z^2 u^{13} + \beta^6 Y^{15} Z^2 u^7 + \beta^5 Y^{16} Z^2 u^{19} + \beta^5 Y^{16} Z^2 u^7\\
                &+\beta^4 Y^{17} Z^2 u^{19} + \beta^4 Y^{17} Z^2 u^{13} + \beta^3 Y^{18} Z^2 u^{13} + \beta^3 Y^4 Z^{16} u^9 + \beta^2 Y^{19} Z^2 u^{19}\\
                &+\beta^2 Y^5 Z^{16} u^9 + \beta Y^{20} Z^2 u^{19} + \beta Y^6 Z^{16} u^9,\\
a_{2}(\beta,Y,Z)&:=\beta^{18} Z^4 u^{23} + \beta^{18} Z^4 u^{20} + \beta^{18} Z^4 u^{11} + \beta^{18} Z^4 u^8 + \beta^{16} Y^2 Z^4 u^{23}\\
                &+ \beta^{16} Y^2 Z^4 u^{20} + \beta^{16} Y^2 Z^4 u^{11} + \beta^{16} Y^2 Z^4 u^8 + \beta^{15} Y^3 Z^4 u^{23} + \beta^{15} Y^3 Z^4 u^{11}\\
                &+ \beta^{14} Y^4 Z^4 u^{23} + \beta^{14} Y^4 Z^4 u^{11} + \beta^{13} Y^5 Z^4 u^{23} + \beta^{13} Y^5 Z^4 u^{11} + \beta^{11} Y^7 Z^4 u^{23}\\
                &+ \beta^{11} Y^7 Z^4 u^{11} + \beta^{10} Y^8 Z^4 u^{23} + \beta^{10} Y^8 Z^4 u^{11} + \beta^9 Y^9 Z^4 u^{23} + \beta^9 Y^9 Z^4 u^{17}\\
                &+ \beta^9 Y^2 Z^{11} u^{18} + \beta^9 Y^2 Z^{11} u^6 + \beta^8 Y^{10} Z^4 u^{17} + \beta^8 Y^{10} Z^4 u^{11} + \beta^8 Y^3 Z^{11} u^{18}\\
                &+ \beta^8 Y^3 Z^{11} u^6 + \beta^7 Y^{11} Z^4 u^{23} + \beta^7 Y^{11} Z^4 u^{11} + \beta^6 Y^{12} Z^4 u^{23} + \beta^6 Y^{12} Z^4 u^{17}\\
                &+ \beta^6 Y^{12} Z^4 u^{11} + \beta^6 Y^{12} Z^4 u^8 + \beta^6 Y^5 Z^{11} u^{18} + \beta^6 Y^5 Z^{11} u^6 + \beta^5 Y^{13} Z^4 u^{23}\\
                &+ \beta^5 Y^{13} Z^4 u^{11} + \beta^4 Y^{14} Z^4 u^{17} + \beta^4 Y^{14} Z^4 u^8 + \beta^4 Y^7 Z^{11} u^{18} + \beta^4 Y^7 Z^{11} u^6\\
                &+ \beta^3 Y^{15} Z^4 u^{23} + \beta^3 Y^{15} Z^4 u^{11} + \beta^3 Y^8 Z^{11} u^6 + \beta^3 Y Z^{18} u^{13} + \beta^2 Y^{16} Z^4 u^{20}\\
                &+ \beta^2 Y^{16} Z^4 u^{17} + \beta^2 Y^9 Z^{11} u^{18} + \beta^2 Y^2 Z^{18} u^{13} + \beta Y^{17} Z^4 u^{23} + \beta Y^{17} Z^4 u^{17}\\
                &+ \beta Y^{10} Z^{11} u^{18} + \beta Y^3 Z^{18} u^{13} + Y^{18} Z^4 u^{23} + Y^{18} Z^4 u^{20} + Y^4 Z^{18} u^{13} + Y^4 Z^{18} u^{10},\\
a_{3}(\beta,Y,Z)&:=\beta^{15} Z^6 u^{21} + \beta^{15} Z^6 u^{15} + \beta^{15} Z^6 u^9 + \beta^{15} Z^6 u^3 + \beta^{14} Y Z^6 u^{21} + \beta^{14} Y Z^6 u^{15}\\
                &+\beta^{14} Y Z^6 u^9 + \beta^{14} Y Z^6 u^3 + \beta^{13} Y^2 Z^6 u^{21} + \beta^{13} Y^2 Z^6 u^{15} + \beta^{13} Y^2 Z^6 u^9\\
                &+\beta^{13} Y^2 Z^6 u^3 + \beta^{11} Y^4 Z^6 u^{21} + \beta^{11} Y^4 Z^6 u^{15} + \beta^{11} Y^4 Z^6 u^9 + \beta^{11} Y^4 Z^6 u^3\\
                &+\beta^{10} Y^5 Z^6 u^{21} + \beta^{10} Y^5 Z^6 u^{15} + \beta^{10} Y^5 Z^6 u^9 + \beta^{10} Y^5 Z^6 u^3 + \beta^9 Y^6 Z^6 u^{15}\\
                &+\beta^9 Y^6 Z^6 u^3 + \beta^8 Y^7 Z^6 u^{21} + \beta^8 Y^7 Z^6 u^9 + \beta^7 Y^8 Z^6 u^{21} + \beta^7 Y^8 Z^6 u^{15}+ \beta^7 Y^8 Z^6 u^9\\
                &+\beta^7 Y^8 Z^6 u^3 + \beta^6 Y^9 Z^6 u^{15} + \beta^6 Y^9 Z^6 u^3 + \beta^5 Y^{10} Z^6 u^{21} + \beta^5 Y^{10} Z^6 u^{15}\\
                &+\beta^5 Y^{10} Z^6 u^9 + \beta^5 Y^{10} Z^6 u^3 + \beta^4 Y^{11} Z^6 u^{21} + \beta^4 Y^{11} Z^6 u^9 + \beta^3 Y^{12} Z^6 u^{21}\\
                &+ \beta^3 Y^{12} Z^6 u^{15} + \beta^3 Y^{12} Z^6 u^9 + \beta^2 Y^{13} Z^6 u^{15} + \beta Y^{14} Z^6 u^{15},\\
a_{4}(\beta,Y,Z)&:=\beta^{15} Y^4 Z u^{18} + \beta^{15} Y^4 Z u^6 + \beta^{14} Y^5 Z u^{18} + \beta^{14} Y^5 Z u^6 + \beta^{13} Y^6 Z u^{18}\\
                &+ \beta^{13} Y^6 Z u^6 + \beta^{12} Z^8 u^{19} + \beta^{12} Z^8 u^{16} + \beta^{12} Z^8 u^{13} + \beta^{12} Z^8 u^{10}\\
                &+\beta^{11} Y^8 Z u^{18} + \beta^{11} Y^8 Z u^6 + \beta^{10} Y^9 Z u^{18} + \beta^{10} Y^9 Z u^6 + \beta^9 Y^{10} Z u^{18}\\
                &+ \beta^9 Y^{10} Z u^6 + \beta^9 Y^3 Z^8 u^{19} + \beta^9 Y^3 Z^8 u^{13} + \beta^8 Y^4 Z^8 u^{16} + \beta^8 Y^4 Z^8 u^{10}\\
                &+ \beta^7 Y^{12} Z u^{18} + \beta^7 Y^{12} Z u^6 + \beta^6 Y^{13} Z u^{18} + \beta^6 Y^{13} Z u^6 + \beta^6 Y^6 Z^8 u^{19}\\
                &+ \beta^6 Y^6 Z^8 u^{10} + \beta^5 Y^{14} Z u^{18} + \beta^5 Y^{14} Z u^6 + \beta^4 Y^8 Z^8 u^{16} + \beta^4 Y^8 Z^8 u^{13}\\
                &+ \beta^3 Y^{16} Z u^{18} + \beta^3 Y^2 Z^{15} u^{14} + \beta^3 Y^2 Z^{15} u^8 + \beta^2 Y^{17} Z u^{18} + \beta^2 Y^{10} Z^8 u^{19}\\
                &+ \beta^2 Y^{10} Z^8 u^{10} + \beta^2 Y^3 Z^{15} u^{14} + \beta^2 Y^3 Z^{15} u^8 + \beta Y^{18} Z u^{18} + \beta Y^{11} Z^8 u^{19}\\
                &+ \beta Y^{11} Z^8 u^{13} + \beta Y^4 Z^{15} u^{14} + \beta Y^4 Z^{15} u^8 + Y^{12} Z^8 u^{19} + Y^{12} Z^8 u^{16},\\
a_{5}(\beta,Y,Z)&:=\beta^9 Z^{10} u^{11} + \beta^9 Z^{10} u^5 + \beta^8 Y Z^{10} u^{11} + \beta^8 Y Z^{10} u^5 + \beta^6 Y^3 Z^{10} u^{11}\\
                &+ \beta^6 Y^3 Z^{10} u^5 + \beta^4 Y^5 Z^{10} u^{11} + \beta^4 Y^5 Z^{10} u^5 + \beta^3 Y^6 Z^{10} u^5 + \beta^2 Y^7 Z^{10} u^{11}\\
                &+ \beta Y^8 Z^{10} u^{11}\\
a_{6}(\beta,Y,Z)&:=\beta^9 Y^4 Z^5 u^{20} + \beta^9 Y^4 Z^5 u^8 + \beta^8 Y^5 Z^5 u^{20} + \beta^8 Y^5 Z^5 u^8 + \beta^6 Y^7 Z^5 u^{20}\\
                &+\beta^6 Y^7 Z^5 u^8 + \beta^6 Z^{12} u^{15} + \beta^6 Z^{12} u^{12} + \beta^6 Z^{12} u^9 + \beta^6 Z^{12} u^6 + \beta^4 Y^9 Z^5 u^{20}\\
                &+ \beta^4 Y^9 Z^5 u^8 + \beta^4 Y^2 Z^{12} u^{15} + \beta^4 Y^2 Z^{12} u^{12} + \beta^4 Y^2 Z^{12} u^9 + \beta^4 Y^2 Z^{12} u^6\\
                &+\beta^3 Y^{10} Z^5 u^8 + \beta^3 Y^3 Z^{12} u^9 + \beta^2 Y^{11} Z^5 u^{20} + \beta^2 Y^4 Z^{12} u^{15} + \beta^2 Y^4 Z^{12} u^{12}\\
                &+\beta^2 Y^4 Z^{12} u^6 + \beta Y^{12} Z^5 u^{20} + \beta Y^5 Z^{12} u^9 + Y^6 Z^{12} u^{15} + Y^6 Z^{12} u^{12},\\
a_{7}(\beta,Y,Z)&:=\beta^{15} Y^2 u^{17} + \beta^{15} Y^2 u^5 + \beta^{14} Y^3 u^{17} + \beta^{14} Y^3 u^5 + \beta^{13} Y^4 u^{17} + \beta^{13} Y^4 u^5\\
                &+\beta^{11} Y^6 u^{17} + \beta^{11} Y^6 u^5 + \beta^{10} Y^7 u^{17} + \beta^{10} Y^7 u^5 + \beta^9 Y^8 u^{17} + \beta^9 Y^8 u^5\\
                &+\beta^7 Y^{10} u^{17} + \beta^7 Y^{10} u^5 + \beta^6 Y^{11} u^{17} + \beta^6 Y^{11} u^5 + \beta^5 Y^{12} u^{17} + \beta^5 Y^{12} u^5\\
                &+\beta^3 Y^{14} u^{17} + \beta^3 Z^{14} u^{13} + \beta^2 Y^{15} u^{17} + \beta^2 Y Z^{14} u^{13} + \beta Y^{16} u^{17}\\
                &+ \beta Y^2 Z^{14} u^{13},\\
a_{8}(\beta,Y,Z)&:=\beta^6 Y^8 Z^2 u^{15} + \beta^6 Y^8 Z^2 u^{12} + \beta^4 Y^{10} Z^2 u^{15} + \beta^4 Y^{10} Z^2 u^{12} + \beta^3 Y^{11} Z^2 u^{15}\\
                &+\beta^3 Y^4 Z^9 u^{16} + \beta^3 Y^4 Z^9 u^{10} + \beta^2 Y^{12} Z^2 u^{12} + \beta^2 Y^5 Z^9 u^{16} + \beta^2 Y^5 Z^9 u^{10}\\
                &+\beta Y^{13} Z^2 u^{15} + \beta Y^6 Z^9 u^{16} + \beta Y^6 Z^9 u^{10},\\
a_{9}(\beta,Y,Z)&:=\beta^9 Y^2 Z^4 u^{19} + \beta^9 Y^2 Z^4 u^7 + \beta^8 Y^3 Z^4 u^{19} + \beta^8 Y^3 Z^4 u^7 + \beta^6 Y^5 Z^4 u^{19}\\
                &+\beta^6 Y^5 Z^4 u^7 + \beta^4 Y^7 Z^4 u^{19} + \beta^4 Y^7 Z^4 u^7 + \beta^3 Y^8 Z^4 u^{13} + \beta^3 Y^8 Z^4 u^7\\
                &+\beta^2 Y^9 Z^4 u^{19} + \beta^2 Y^9 Z^4 u^{13} + \beta Y^{10} Z^4 u^{19} + \beta Y^{10} Z^4 u^{13},\\
a_{{10}}(\beta,Y,Z)&:=\beta^3 Y^5 Z^6 u^{17} + \beta^2 Y^6 Z^6 u^{17} + \beta Y^7 Z^6 u^{17} + Y^8 Z^6 u^{17} + Y^8 Z^6 u^{14},\\
a_{{11}}(\beta,Y,Z)&:=\beta^3 Y^2 Z^8 u^9 + \beta^2 Y^3 Z^8 u^9 + \beta Y^4 Z^8 u^9,\\
a_{{12}}(\beta,Y,Z)&:=\beta^9 Z^3 u^{18} + \beta^9 Z^3 u^{12} + \beta^8 Y Z^3 u^{18} + \beta^8 Y Z^3 u^{12} + \beta^6 Y^3 Z^3 u^{18} + \beta^6 Y^3 Z^3 u^{12}\\
                    &+\beta^4 Y^5 Z^3 u^{18} + \beta^4 Y^5 Z^3 u^{12} + \beta^3 Y^6 Z^3 u^{18} + \beta^2 Y^7 Z^3 u^{12} + \beta Y^8 Z^3 u^{12},\\
a_{{13}}(\beta,Y,Z)&:=0,\\
a_{{14}}(\beta,Y,Z)&:=\beta^9 Y u^{19} + \beta^9 Y u^{13} + \beta^8 Y^2 u^{19} + \beta^8 Y^2 u^{13} + \beta^6 Y^4 u^{19} + \beta^6 Y^4 u^{10}\\
                &+ \beta^4 Y^6 u^{19} + \beta^4 Y^6 u^{10} + \beta^3 Y^7 u^{13} + \beta^3 Z^7 u^{14} + \beta^2 Y^8 u^{19} + \beta^2 Y^8 u^{13}\\
                &+ \beta^2 Y^8 u^{10} + \beta^2 Y Z^7 u^{14} + \beta Y^9 u^{19} + \beta Y^2 Z^7 u^{14},\\
a_{{15}}(\beta,Y,Z)&:=\beta^3 Y^4 Z^2 u^{17} + \beta^3 Y^4 Z^2 u^{11} + \beta^2 Y^5 Z^2 u^{17} + \beta^2 Y^5 Z^2 u^{11} + \beta Y^6 Z^2 u^{17} + \beta Y^6 Z^2 u^{11},\\
a_{16}(\beta,Y,Z)&:=\beta^3 Y Z^4 u^{15} + \beta^2 Y^2 Z^4 u^{15} + \beta Y^3 Z^4 u^{15} + Y^4 Z^4 u^{15} + Y^4 Z^4 u^{12},\\
a_{{17}}(\beta,Y,Z)&:=0,\\
a_{{18}}(\beta,Y,Z)&:=\beta^3 Y^2 Z u^{16} + \beta^2 Y^3 Z u^{16} + \beta Y^4 Z u^{16},\\
a_{{19}}(\beta,Y,Z)&:=0,\\
a_{{20}}(\beta,Y,Z)&:=0,\\
a_{{21}}(\beta,Y,Z)&:=\beta^3 u^{15} + \beta^2 Y u^{15} + \beta Y^2 u^{15}.
\end{align*}
}


\begin{thebibliography}{99}

\bibitem{MR2648536}
Y. Aubry, G. McGuire, F. Rodier, \emph{A few more functions that are not APN infinitely often}, Contemp. Math. \textbf{518} (2010), 23--31.

\bibitem{BT22}
D. Bartoli, M. Timpanella, \emph{On a conjecture on APN permutations}, Cryptogr. Commun. \textbf{14} (2022), 925--931.

\bibitem{MR4414815}
C. Beierle, C. Carlet, G. Leander, L. Perrin, \emph{A further study of quadratic APN permutations in dimension nine}, Finite Fields Appl. \textbf{81} (2022), 102049.

\bibitem{BL22}
C. Beierle, G. Leander, \emph{New instances of quadratic APN functions}, IEEE Trans. Inf. Theory \textbf{68} (2022), 670--678.

\bibitem{BS91}
E. Biham, A. Shamir, \emph{Differential cryptanalysis of DES-like cryptosystems}, J. Cryptol. \textbf{4} (1991), 3--72.

\bibitem{MR1484478}
W. Bosma, J. Cannon, C. Playoust, \emph{The Magma algebra system I: The user language}, J. Symbolic Comput. \textbf{24} (1997), 235--265.

\bibitem{MR2206396}
A. Cafure, G. Matera, \emph{Improved explicit estimates on the number of solutions of equations over a finite field}, Finite Fields Appl. \textbf{12} (2006), 155--185.

\bibitem{Ha77}
R. Hartshorne, \emph{Algebraic Geometry}, Graduate Texts in Mathematics 52, Springer-Verlag, 1977.

\bibitem{HKT13}
J.W.P. Hirschfeld, G. Korchm\'aros, F. Torres, \emph{Algebraic Curves over a Finite Field}, Princeton University Press, 2013.

\bibitem{LK24}
K. Li, N. Kaleyski, \emph{Two new infinite families of APN functions in trivariate form}, IEEE Trans. Inf. Theory \textbf{70} (2024), 1436--1452.

\bibitem{GithubPS25}
P. St\u anic\u a, \emph{Trivariate permutation project}, source code, \texttt{https://github.com/pstanica/trivariatePP-APN}, 2025.

\end{thebibliography}
\end{document}